\documentclass[11pt]{amsart}
\usepackage[twoside,left=3.66cm,right=3.66cm,bottom=3cm,top=3cm]{geometry}
\usepackage{amscd,amsmath,amssymb,amsfonts,bbm}
\usepackage[T1]{fontenc}
\usepackage{lmodern}
\usepackage{mathtools}
\usepackage{enumerate}
\usepackage{multicol}
\usepackage[all]{xy}
\usepackage{footmisc}

\newtheorem{thm}{Theorem}[section]
\newtheorem{prop}[thm]{Proposition}

\newtheorem{cor}[thm]{Corollary}

\theoremstyle{definition}

\theoremstyle{remark}
\newtheorem{rem}[thm]{Remark}

\numberwithin{equation}{section}

\DeclareMathOperator{\Q}{\mathbb{Q}}
\DeclareMathOperator{\sO}{\mathcal{O}}

\DeclareMathOperator{\bP}{\mathbb{P}}
\DeclareMathOperator{\bA}{\mathbb{A}}

\DeclareMathOperator{\CH}{CH}

\DeclareMathOperator{\K3}{K3}
\DeclareMathOperator{\cl}{cl}
\DeclareMathOperator{\Id}{Id}

\def\Z{\mathbb Z}
\def\C{\mathbb C}

\begin{document}

\title[The integral Hodge conjecture for threefolds of Kodaira dimension $0$]
{Counterexamples to the integral Hodge conjecture for threefolds of Kodaira dimension zero}%
\title[Curve classes on the product of a curve and an Enriques surface]
{Curve classes on the product of a curve and an Enriques surface}%
\title[The integral Hodge conjecture for threefolds of Kodaira dimension $0$]
{Failure of the integral Hodge conjecture for threefolds of Kodaira dimension zero}

\author{Olivier Benoist}
\address{D\'epartement de math\'ematiques et applications, \'Ecole normale sup\'erieure,
45 rue d'Ulm, 75230 Paris Cedex 05, France}
\email{olivier.benoist@ens.fr}

\author{John Christian Ottem}
\address{University of Oslo, Box 1053, Blindern, 0316 Oslo, Norway}
\email{johnco@math.uio.no}

\begin{abstract}
We prove that the product of an Enriques surface and a very general curve of genus at least $1$ does not satisfy the integral Hodge conjecture for $1$-cycles. This provides the first examples of smooth projective complex threefolds of Kodaira dimension zero for which the integral Hodge conjecture fails, and the first examples of non-algebraic torsion cohomology classes of degree $4$ on smooth projective complex threefolds. 
\end{abstract}
\maketitle
\thispagestyle{empty}

\section*{Introduction}\label{intro}

  The integral Hodge conjecture for codimension $k$ cycles on a smooth complex projective variety $X$ asserts that the image of the cycle class map  in Betti cohomology  $\cl:\CH^k(X)\to H^{2k}(X,\Z)$ coincides with the subgroup of integral Hodge classes: those integral cohomology classes that are of type $(k,k)$ in the Hodge decomposition. This statement always holds if $k=0$ or $k=\dim(X)$ for trivial reasons, and for $k=1$ by the Lefschetz $(1,1)$ theorem.

  The first counterexamples to this statement, i.e. the first examples of integral Hodge classes that are not algebraic, were discovered by Atiyah and Hirzeburch \cite{AH}. Much simpler examples, relying on a degeneration argument, were found later by Koll\'ar \cite{Trento}, showing that the integral Hodge conjecture may already fail for $1$-cycles on threefolds.

  Voisin showed that, despite these counterexamples, the integral Hodge conjecture may hold in interesting special cases if one imposes restrictions on the geometry of the variety $X$. For instance, she proved in \cite[Theorem 2]{Voisin3} that it holds for $1$-cycles on threefolds $X$ that are either uniruled, or satisfy $K_X=0$ and $H^2(X,\sO_X)=0$, and Totaro announced a proof that, in the latter case, the hypothesis that $H^2(X,\sO_X)=0$ may be removed. Voisin also conjectured that it should hold for $1$-cycles on rationally connected varieties of arbitrary dimensions \cite[Question 16]{VoisinHC} (see \cite{HV,Floris} for partial results in this direction).

 These positive or conjectural statements are close to be optimal. The integral Hodge conjecture can fail for codimension $2$ cycles on rationally connected varieties \cite[Th\'eor\`eme 1.3]{CTV}, even for fourfolds \cite[Corollary 1.6]{Stefan}. It is also known to fail for $1$-cycles on some threefolds of Kodaira dimension $1$ \cite[Theorem 3.1]{Totaro}, as well as on some threefolds such that $H^i(X,\sO_X)=0$ for all $i>0$ \cite[Proposition 5.7]{CTV}. The case of $1$-cycles on threefolds of Kodaira dimension $0$, raised in \cite[\S 3]{Totaro}, remained open. In this paper, we provide counterexamples in this situation:

\begin{thm}
\label{mainth}
Let $S$ be an Enriques surface over $\C$, and let $g\geq 1$ be an integer. Then, if $B$ is a very general smooth projective curve of genus $g$ over $\C$,  the integral Hodge conjecture for $1$-cycles on $B\times S$ does not hold.
\end{thm}

 In Section \ref{IHCholds}, we show that the integral Hodge conjecture may nevertheless still hold for some particular products of an Enriques surface and an elliptic curve.

\begin{cor}
\label{maincoro}
There exists a smooth projective threefold of Kodaira dimension $0$ over $\C$ that does not satisfy the integral Hodge conjecture for $1$-cycles.
\end{cor}

 Corollary \ref{maincoro} is deduced from Theorem \ref{mainth} by taking $g=1$. Thus, the threefold $X$ that we consider is the product of an Enriques surface and an elliptic curve. It satisfies $2K_X=0$, but $K_X\neq 0$, illustrating the sharpness of Voisin's theorem \cite[Theorem 2]{Voisin3}, and of its  aforementioned improvement by Totaro. In Corollary \ref{allKodairadim}, we explain how to deduce counterexamples to the integral Hodge conjecture in essentially all possible degrees, dimensions and Kodaira dimensions from this result.

\begin{cor}
\label{secondcoro}
There exists a smooth projective threefold $X$ over $\C$ carrying a non-algebraic torsion cohomology class of degree $4$.
\end{cor}

  By Remark \ref{2torsion} below, any variety as in Theorem \ref{mainth} satisfies the condition of Corollary \ref{secondcoro}. Examples of non-algebraic torsion cohomology classes of even degree were constructed by Atiyah and Hirzebruch \cite{AH} in dimension $7$ and by Soul\'e and Voisin \cite[Theorem 3]{SV} in dimension $5$. As far as we know, and as indicated in \cite[p. 2]{Totaro}, no $3$-dimensional example was previously known. 

\vspace{1em}

Let us briefly outline the construction of the counterexamples. Section \ref{sec1} is devoted to understanding what it means for the product $X=S\times B$ of an Enriques surface $S$ and a smooth projective curve $B$ to satisfy the integral Hodge conjecture for 1-cycles.  In Proposition \ref{whatisIHC}, we show that it is equivalent to every finite \'etale cover of degree $2$ of $B$ being induced, by means of an algebraic correspondence, by the $\K3$ cover of $S$.

  To prove Theorem \ref{mainth}, we need to contradict this statement for some curves $B$. We use a degeneration argument: if all finite \'etale covers of degree $2$ of all smooth projective curves of genus $g$ were induced, through a correspondence, by the universal cover of $S$, then the same would be true for degenerations of such covers. But some of these degenerations are ramified (for instance, some of Beauville's admissible covers \cite{Beauville}), giving a contradiction.

The idea of using degeneration techniques to contradict the integral Hodge conjecture was initiated by Koll\'ar \cite{Trento} and further developed by Totaro \cite{Totaro}. Our argument differs from theirs in the sense that it does not rely on the analysis of algebraic cycles on the limit variety. Colliot-Th\'el\`ene pointed out to us that our degeneration argument is similar to one used by Gabber in \cite[Appendix]{CTGabber} to construct unramified Brauer classes on smooth projective varieties whose period and index differ, and that it is possible to give an alternative proof of Theorem \ref{mainth} using the specialization arguments of \emph{loc.\ cit.} and the reformulation of the integral Hodge conjecture in terms of unramified cohomology (see \cite{colliot}).

\subsection*{Acknowledgements} We would like to thank J\o rgen Vold Rennemo, Burt Totaro, Claire Voisin and Olivier Wittenberg for useful discussions. We are also grateful to Jean-Louis Colliot-Th\'el\`ene for drawing our attention to the reference \cite{CTGabber}.

\bigskip

\section{The product of a curve and an Enriques surface}
\label{sec1}

We fix an Enriques surface $S$ over $\C$ and denote by $\alpha\in H^1(S,\Z/2)$ the class corresponding to its degree $2$ finite \'etale cover 
by a $\K3$ surface.
Let $B$ be a smooth projective curve over $\C$. We consider the threefold $X=B\times S$, with projections $p_1:X\to B$ and $p_2:X\to S$.

\begin{prop}
\label{whatisIHC}
The integral Hodge conjecture holds for $1$-cycles on $X$ if and only if for every $\beta\in H^1(B,\Z/2)$, there exists a correspondence $Z\in \CH_1(B\times S)$ such that $Z^*\alpha=\beta$.
\end{prop}

\begin{proof}
 Since $H^1(S,\sO_S)=H^2(S,\sO_S)=0$, the group $H^2(X,\sO_X)$ vanishes, so that the integral Hodge conjecture for $1$-cycles on $X$ is the statement that the cycle class map $\cl:\CH_1(X)\to H^4(X,\Z)$ is surjective. The Betti cohomology of $B$ is torsion free, so using the K\"unneth formula \cite[Theorem 3.16]{Hatcher}, we obtain an isomorphism:
\begin{equation}
\label{decompo}
\bigoplus_{i=0}^2 H^i(B,\Z)\otimes H^{4-i}(S,\Z)\xrightarrow{\sim}H^4(X,\Z),
\end{equation}
given by the formula $(a_i\otimes b_{4-i})_{0\leq i\leq 2} \mapsto \sum_{i=0}^2 p_1^*a_i\smile p_2^* b_{4-i}$.
Since the groups $H^0(B,\Z)$, $H^2(B,\Z)$, $H^4(S,\Z)$ and $H^2(S,\Z)$ are generated by classes of algebraic cycles (the last one by the Lefschetz $(1,1)$ theorem for $S$), the factors 
$H^0(B,\Z)\otimes H^4(S,\Z)$ and $H^2(B,\Z)\otimes H^2(S,\Z)$ in the decomposition (\ref{decompo}) consist of algebraic classes.
Consequently, the validity of the integral Hodge conjecture for $1$-cycles on $X$ is equivalent to the surjectivity of the composition $$\phi:\CH_1(X)\to H^1(B,\Z)\otimes H^{3}(S,\Z)$$ of
$\cl: \CH_1(X)\to H^4(X,\Z)$ and of the projection on the  second factor of the decomposition (\ref{decompo}).
 Since the reduction modulo $2$ gives an isomorphism $H^{3}(S,\Z)\xrightarrow{\sim}H^{3}(S,\Z/2)$, the morphism $\phi$ identifies with the composition
$$\CH_1(X)\to H^4(X,\Z/2)\to H^1(B,\Z/2)\otimes H^3(S,\Z/2)$$
 of the cycle class map $\cl_2:\CH_1(X)\to H^4(X,\Z/2)$ modulo $2$ and of the projection 
$\pi:H^4(X,\Z/2)\to H^1(B,\Z/2)\otimes H^3(S,\Z/2)$ given by the K\"unneth formula with $\Z/2$ coefficients.

The map $H^4(X,\Z)\to H^1(B,\Z)$, given by the formula $\gamma\mapsto p_{1*}(\gamma\smile p_2^*\alpha)$, vanishes on the first and the third factor of the decomposition (\ref{decompo}). On the second factor, it coincides with the isomorphism $\iota:H^1(B,\Z/2)\otimes H^3(S,\Z/2)\xrightarrow{\sim} H^1(B,\Z/2)$ induced by the identification $H^3(S,\Z/2)=\Z/2$, since the cup-product with the generator $\alpha\in H^1(S,\Z/2)=\Z/2$ yields an isomorphism $H^3(S,\Z/2)\xrightarrow{\smile\alpha} H^4(S,\Z/2)=\Z/2$ by Poincar\'e duality.
One computes, for $Z\in \CH_1(X)$:
$$Z^*\alpha=p_{1*}(\cl_2(Z)\smile p_2^*\alpha)=p_{1*}(\pi(\cl_2(Z))\smile p_2^*\alpha)=\iota(\phi(Z))\in H^1(B,\Z/2).$$
The morphism $\phi$ is then surjective if and only if $\iota\circ\phi$ is surjective, if and only if every $\beta\in H^1(B,\Z/2)$ is of the form $Z^*\alpha$ for some $Z\in\CH_1(X)$, as wanted.
\end{proof}

\begin{rem}
\label{2torsion}
It follows from the proof of Proposition \ref{whatisIHC} that if the integral Hodge conjecture for $1$-cycles on $X$ fails, then $H^4(X,\Z)$ contains a $2$-torsion class which is non-algebraic. Indeed, the first and third factors of (\ref{decompo}) are algebraic and the second is $2$-torsion because $H^3(S,\Z)=\Z/2$.
\end{rem}

\section{Degeneration to a nodal curve}
\label{sec2}
As in the previous section, we consider an Enriques surface $S$ over $\C$ and let $\alpha\in H^1(S,\Z/2)$ denote the class corresponding to its $\K3$ cover.

In the proof of Theorem \ref{mainth}, we will use a degeneration of finite \'etale double covers of elliptic curves. We construct this degeneration as follows. Let $T=\bA^1\setminus \{1\}$ and $\mathcal{E}\subset \bP^2\times T$ be the Legendre family of elliptic curves defined by the equation $$Y^2Z=X(X-Z)(X-tZ),$$ 
where $X,Y,Z$ are homogenous coordinates in $\bP^2$ and $t$ is the coordinate of $\bA^1$. The fibers of the second projection $f:\mathcal{E}\to T$ are smooth elliptic curves except for the one above $0\in T$, that is a nodal rational curve, with one singular point $x=[0:0:1]\in \mathcal{E}_0$. 

The two constant sections of $f$ given by $[0:1:0]$  and $[1:0:1]$ are the identity and a $2$-torsion point for the group law on the $T$-group scheme $\mathcal{E}\setminus\{x\}\to T$. Since $\mathcal{E}$ is regular and $f$-minimal, the translation by this $2$-torsion section extends to an involution  $i$ of $\mathcal{E}$ (see \cite[Chapter 9, Proposition 3.13]{Liu}). Let $q:\mathcal{E}\to \mathcal{F}$ be the quotient by $i$, and let $g:\mathcal{F}\to T$ be the induced morphism. Since $x$ is the only fixed point of $i$, the quotient map $q$ is \'etale outside of $x$. If $t\neq 0$, the curve $\mathcal{F}_t$ is elliptic as a fixed point free quotient of the elliptic curve $\mathcal{E}_t$. Being dominated by a rational curve, the special fiber $\mathcal{F}_0$ is itself a (singular) rational curve.

\begin{prop}
\label{elliptic}
There exists an elliptic curve $E$ over $\C$ such that the integral Hodge conjecture for $1$-cycles on $E\times S$ does not hold.
\end{prop}

\begin{proof}
 Let $t\in T$ be a very general point and suppose for contradiction that the integral Hodge conjecture for $1$-cycles on $\mathcal{F}_t\times S$ holds. Applying Proposition \ref{whatisIHC} to the class $\beta\in H^1(\mathcal{F}_t,\Z/2)$ of the finite \'etale double cover $q_t:\mathcal{E}_t\to\mathcal{F}_t$ shows that there exists $Z\in \CH_1(\mathcal{F}_t\times S)$ such that $Z^*\alpha=\beta\in H^1(\mathcal{F}_t,\Z/2)$.

Since $t$ is very general, all the irreducible components of the relative Hilbert scheme of $g\circ pr_1:\mathcal{F}\times S\to T$ whose images contain $t$ dominate $T$. This applies to the components of this relative Hilbert scheme parametrizing the components of the support of $Z$. Since the components of the relative Hilbert scheme are moreover proper over $T$, it follows that there exists a finite surjective morphism $\mu:T'\to T$ of smooth integral curves with the following property. Let $t'\in T'$ be a preimage of  $t\in T$ by $\mu$, and let $g':\mathcal{F}'\to T'$ and $q':\mathcal{E}'\to \mathcal{F}'$ be the morphisms obtained from $g$ and $q$ by base change by $\mu$. Then there exists a cycle $\mathcal{Z}\in \CH^2(\mathcal{F}'\times S)$ all of whose components dominate $T'$ such that $\mathcal{Z}|_{\mathcal{F}'_{t'}\times S}=Z$.

Let $U=\mu^{-1}(T\setminus\{0\})\subset T'$. We consider the two following finite \'etale double covers of the smooth variety $\mathcal{F}'_U=g'^{-1}(U)$. The first is the restriction $\mathcal{E}'_U\to \mathcal{F}'_U$ of $q'$. The second, that we will denote by $\mathcal{G}'_U\to\mathcal{F}'_U$, is the one associated to the class $(\mathcal{Z}|_{\mathcal{F}'_U\times S})^*\alpha\in H^1(\mathcal{F}'_U,\Z/2)$.
The classes of these two double covers coincide in $H^1(\mathcal{F}'_{t'},\Z/2)$ because 
$(\mathcal{Z}|_{\mathcal{F}'_U\times S})^*\alpha|_{\mathcal{F}'_{t'}}=Z^*\alpha=\beta$. Since $U$ is connected and $R^1g'_*\Z/2$ is locally constant on $U$, the exact sequence
$$0\to H^1(U,\Z/2)\to H^1( \mathcal{F}'_U,\Z/2)\to H^0(U, R^1g'_*\Z/2)$$
induced by the Leray spectral sequence of $g'$ shows that 
$[\mathcal{E}'_U]-[\mathcal{G}'_U]$ belongs to the subgroup $H^1(U,\Z/2)$ of $H^1( \mathcal{F}'_U,\Z/2)$.
Consequently, up to replacing $T'$ by a finite double cover (maybe ramified outside of $U$), we may assume that 
$\mathcal{E}'_U$ and $\mathcal{G}'_U$ are isomorphic as finite \'etale double covers of $\mathcal{F}'_U$.

 Pick a preimage of $0\in T$ under $\mu:T'\to T$, and denote it, for simplicity, also by $0\in T'$. 
Let $\widetilde{\mathcal{F}}\to \mathcal{F}'$ be a resolution of singularities which is an isomorphism over $\mathcal{F}'_U$, and let $\widetilde{\mathcal{Z}}\in  \CH^2(\widetilde{\mathcal{F}}\times S)$ denote the strict transform of $\mathcal{Z}$. 
Let $\widetilde{q}:\widetilde{\mathcal{E}}\to\widetilde{\mathcal{F}}$ be the base change of $q'$, and let $\widetilde{\mathcal{G}}\to\widetilde{\mathcal{F}}$ be the finite \'etale double cover of $\widetilde{\mathcal{F}}$ corresponding to the class $\widetilde{\mathcal{Z}}^*\alpha\in H^1(\widetilde{\mathcal{F}},\Z/2\Z)$. We have shown above that the two finite covers $\widetilde{\mathcal{E}}$ and $\widetilde{\mathcal{G}}$ of $\widetilde{\mathcal{F}}$ 
are isomorphic over $\mathcal{F}'_U$. By \cite[Exp. I, Corollaire 10.3]{SGA1}, it follows that they coincide over the locus where  $\widetilde{q}$ is \'etale. 
In particular, they coincide at the generic point $\eta$ of the strict transform of $\mathcal{F}'_0$ in $\widetilde{\mathcal{F}}$. However, the cover $\widetilde{\mathcal{G}}\to\widetilde{\mathcal{F}}$ splits at $\eta$ since it is finite \'etale and the normalization of $\mathcal{F}'_0$ is isomorphic to $\bP^1$, whereas 
$\widetilde{q}:\widetilde{\mathcal{E}}\to\widetilde{\mathcal{F}}$ does not split at $\eta$ because the fiber $\mathcal{E}'_0=\mathcal{E}_0$ is irreducible. This is the required contradiction.
\end{proof}

We may now prove our main result:

\begin{proof}[Proof of Theorem \ref{mainth}]
Let $E$ be an elliptic curve as in Proposition \ref{whatisIHC}, and let $g\geq 1$ be an integer. Choose a smooth projective curve $B$ of genus $g$ that admits a morphism $\rho:B\to E$ of odd degree $\delta$.
By Remark \ref{2torsion}, there exists a $2$-torsion cohomology class $\omega\in H^4(E\times S,\Z)$ which is not algebraic. The pullback $(\rho,\Id)^*\omega\in H^4(B\times S,\Z)$ is also 2-torsion and not algebraic; if it were, then the same would be true for $(\rho,\Id)_*(\rho,\Id)^*\omega=\delta\omega=\omega$, a contradiction. Consequently, the integral Hodge conjecture for $1$-cycles on $B\times S$ does not hold.

\def\B{\mathcal B}
To prove the statement for the very general curve, we use a specialization argument similar to that in \cite{Trento} and \cite[\S 6.1.2]{Voisinorange}. To give some details, let $T$ be a connected component of the moduli space of smooth genus $g$ curves over $\C$ with level $4$ structure, and let $\pi:\B\to T$ be its universal family. As the moduli space of smooth curves of genus $g$ is connected, every smooth curve of genus $g$ appears as a fiber of $\pi$.
 We say that a point $t\in T$ is very general if every irreducible component of the relative Hilbert scheme of $\mathcal{B}\times S\to T$ whose image in $T$ contains $t$ dominates $T$. Arguing as in the proof of Proposition \ref{elliptic}
shows that if the integral Hodge conjecture holds for some $\B_t\times S$ with $t\in T$ very general, then it holds for $\B_t\times S$ for all $t\in T$, which contradicts the conclusion of the above paragraph. 
\end{proof}


\begin{rem}
 Combining our argument with the idea of Hassett and Tschinkel to use specializations in characteristic $p$ (see \cite[Remarque 5.9]{CTV} and \cite{Totaro}) yields explicit examples of varieties as in Theorem \ref{mainth} or Corollary \ref{maincoro}, that are moreover defined over $\Q$. More precisely, if $S$ is a complex Enriques surface that is defined over $\Q$, one may choose the elliptic curve in Proposition \ref{elliptic} to have equation $Y^2Z=X(X-Z)(X-pZ)$ for any odd prime number $p$ of good reduction of $S$.
\end{rem}

\begin{cor}
\label{allKodairadim}
For any $n\ge 3$, any $\kappa\in \{-\infty,0,1,\ldots,n\}$ (with $\kappa\ge 0$ if $n=3$) and any $k\in\{2,\dots,n-1\}$, there is a smooth projective variety $X$ of dimension $n$ and Kodaira dimension $\kappa$ such that the integral Hodge conjecture fails for codimension $k$ cycles on $X$.
\end{cor}

\begin{proof}
Let $S$, $E$ and $B$ be as in the proof of Theorem \ref{mainth}, with $B$ of genus $g\geq 2$. Let $\Sigma$ be a smooth projective surface of general type admitting an odd degree morphism to $S$. Arguing as above, we see that the threefolds $E\times S$, $B\times S$, $E\times \Sigma$ and $B\times \Sigma$ carry $2$-torsion non-algebraic cohomology classes of degree $4$. Their Kodaira dimensions are respectively $0$, $1$, $2$ and $3$. 

Let $Y$ be one of the above threefolds and let $\omega\in H^4(Y,\Z)$ be a non-algebraic $2$-torsion class.
Let $C_1,\dots,C_{n-3}$ be smooth projective connected curves, and let $\omega_i\in H^2(C_i,\Z)$ be the class of a point. Define $X=Y\times \prod_i C_i$ with projections $q:X\to Y$ and $p_i:X\to C_i$. The $2$-torsion class $\gamma=q^*\omega\smile \cup_{i=1}^{k-2} p_i^*\omega_i$ is not algebraic because so would be $\omega=q_*(\gamma\smile\cup_{i=k-1}^{n-3}p_i^*\omega_i)$. Choosing $Y$ and the genera of the $C_i$ appropriately produces all the required counterexamples.
\end{proof}

\section{Products for which the integral Hodge conjecture holds}
\label{IHCholds}
In light of Theorem \ref{mainth}, it is natural to ask whether the integral Hodge conjecture in fact fails on every product of an Enriques surface and an elliptic curve. In this section, we show that this is not the case, by constructing explicit examples of $S$ and $E$ such that every class of $H^1(E,\mathbb Z/2)$ satisfies the condition in Proposition \ref{whatisIHC}.

We consider an Enriques surface $S$ admitting an elliptic fibration with a double fiber whose reduction is an elliptic curve $E$. In this case, $E$ does not split in the $\K3$ cover of $S$  (see \cite[Chapter VIII.17]{BPVdV}), and so there exists a non-zero class $\beta\in H^1(E,\Z/2\Z)$ satisfying the hypothesis of Proposition \ref{whatisIHC}. If we choose $S$ so that $E$ is  the elliptic curve $\{Y^2Z=X^3+Z^3\}\subset\bP^2$, the automorphisms of $E$ act transitively on the non-zero elements of $H^1(E,\Z/2)$. It follows that all of these classes satisfy the hypothesis of Proposition \ref{whatisIHC}, and consequently the integral Hodge conjecture holds for $1$-cycles on $E\times S$.

We can construct such $S$ and $E$ explicitly using the construction described in \cite[Chapter V.23]{BPVdV}. There, Enriques surfaces are constructed as quotients by a fixed point free involution of the minimal resolution of a double cover of $\bP^1\times \bP^1$ branched along an invariant $(4,4)$-curve. The projection on one of the factors induces an elliptic fibration on the Enriques surface. It is easily seen from the equations of  \emph{loc.\ cit.} that any elliptic curve may appear as the reduction of a double fiber of an elliptic fibration arising from this construction.


\bibliographystyle{myamsalpha}
\bibliography{IHCEnriques}

\end{document}